\documentclass[12pt]{amsart}
\usepackage{amssymb,verbatim,amscd,amsmath,graphicx,enumerate}
\usepackage{amsmath}
\usepackage{graphicx}
\usepackage{caption}
\usepackage{subcaption}
\usepackage{pdfsync}
\usepackage{fullpage}
\usepackage{color}
\usepackage{enumitem}
\usepackage{marginnote}
\usepackage{tikz}

\usetikzlibrary{patterns}
\usetikzlibrary{calc}
\usetikzlibrary{matrix}
\usepgflibrary{arrows}
\usetikzlibrary{arrows}
\usetikzlibrary{decorations.pathreplacing}
\usetikzlibrary{decorations.pathmorphing}
\usepackage{hyperref}

\numberwithin{equation}{section}
\newtheorem{theorem}{Theorem}[section]
\newtheorem{lemma}[theorem]{Lemma}
\newtheorem{proposition}[theorem]{Proposition}
\newtheorem{corollary}[theorem]{Corollary}

\theoremstyle{definition}
\newtheorem{definition}[theorem]{Definition}

\newtheorem{def-prop}[theorem]{Definition-Proposition}
\newtheorem{remark}[theorem]{Remark}
\newtheorem{example}[theorem]{Example}

\newtheorem*{Mysketch}{Sketch of proof} % \newtheorem establishes the object heading
  {\pushQED{\qed}\begin{Mysketch}}
  {\popQED\end{Mysketch}}

\DeclareMathOperator{\Ass}{Ass}
\DeclareMathOperator{\Min}{Min}

\DeclareMathOperator{\depth}{depth}

\DeclareMathOperator{\Ext}{Ext}

\DeclareMathOperator{\ini}{in_{>}}
\DeclareMathOperator{\pol}{pol}
\DeclareMathOperator{\depol}{depol}

\newcommand{\lcm}[1]{\ensuremath{{\rm{lcm}}{#1}}}

\begin{document}

\title{Initially Regular Sequences on Cycles and Depth of Unicyclic Graphs}

\author{Le Tran}
\address{Department of Mathematical Sciences \\
New Mexico State University\\
P.O. Box 30001 \\
Department 3MB \\
Las Cruces, NM 88003}
\email{letran95@nmsu.edu}
\urladdr{ https://sites.google.com/view/tran-ngocle}


\begin{abstract}
In this article, we establish initially regular sequences on cycles of the form $C_{3n+2}$ for $n\ge 1$, in the sense of \cite{FHM-ini}. These sequences accurately compute the depth of these cycles, completing the case of finding effective initially regular sequences on cycles. Our approach involves a careful analysis of associated primes of initial ideals of the form $\ini(I,f)$ for arbitrary monomial ideals $I$ and $f$ linear sums. We describe the minimal associated primes of these ideals in terms of the minimal primes of $I$. Moreover, we obtain a description of the embedded associated primes of arbitrary monomial ideals. Finally, we accurately compute the depth of certain types of unicyclic graphs.
\end{abstract}

\maketitle

\section{Introduction}\label{intro}

Let $R$ be a Noetherian local ring, let $\mathfrak{m}$ be the unique maximal ideal of $R$, and let $k = R/\mathfrak{m}$ denote the residue field of $R$. For a finitely generated module $M$, the \textit{depth} of $M$ is an invariant that has been used extensively in the study of rings and modules. In general, $\depth M$ is bounded above by the dimension of $M$ and when equality is achieved, then the module is Cohen-Macaulay. This an important class of modules that has many desirable properties. The depth of a module was first introduced in a homological setting, namely 
$$\depth(M) = \min\left\{i \mid \Ext^i_R(k,M) \ne 0\right\}.$$

In the Noetherian setting, we can also compute the depth of a module using the notion of a maximal $M$-\textit{regular sequence} contained in $\mathfrak{m}$. A sequence $f_1, \ldots, f_d\in \mathfrak{m}$ is said to be an \textit{$M$-regular sequence} if $f_1$ is a \textit{nonzero divisor} (\textit{regular}) on $M$ and $f_i$ is regular on $M/(f_1,\ldots, f_{i-1})M$ for $2\le i\le d$. An $M$-regular sequence $f_1,\ldots, f_d\in \mathfrak{m}$ is maximal if for any $f_{d+1}\in\mathfrak{m}$, the sequence $f_1,\ldots,f_{d+1}$ is not an $M$-regular sequence. It can be seen that the depth of $M$ is the length of any maximal $M$-regular sequence in $\mathfrak{m}$. In fact, all maximal $M$-regular sequences have the same length, so one can take advantage of this to compute the depth of $M$. However, identifying an $M$-regular sequence is in general difficult.

In this work, we focus our attention on the special class of monomial ideals, and in particular squarefree monomial ideals. We can take advantage of the underlying combinatorial structures to determine certain types of sequences, namely initially regular sequences as in \cite{FHM-ini}, that in turn give a lower bound on the depth of modules of the form $R/I$, where $I$ is a monomial ideal in a polynomial ring $R$. 

Let $R$ be a polynomial ring over a field and $I$ a monomial ideal in $R$. In general, it is difficult to obtain the value of $\depth(R/I)$, or more generally $\depth(R/I^t)$, for $t\ge 1$. The function $F(t)=\depth (R/I^t)$ for $t\ge 1$ is called the \textit{depth function} of $I$. A classic result of Burch \cite{Burch} that was improved by  Brodmann \cite{Brodmann} establishes the limiting behavior of the depth function, namely  $\lim_{t\to\infty}\depth(R/I^t) \le \dim R - \ell(I)$, where $\ell(I)$ is the analytic spread of $I$. It was shown by Eisenbud and Huneke \cite{Eis-Hun} that equality holds if the associated graded ring $gr_R(I) = \displaystyle\bigoplus_{i=1}^\infty I^i/I^{i+1}$ of $I$ is Cohen-Macaulay.  However, the initial behavior of $\depth(R/I^t)$ is still mysterious and is sometimes wild, see \cite{Bandari et. al.}.

It is natural to focus on finding lower bounds for $\depth(R/I^t)$. Herzog and Hibi showed that $\depth(R/I^t)$ is a decreasing function if all powers of $I$ have a linear resolution, \cite{powerofidealHerzog}. In general, edge ideals and their powers do not have linear resolutions. Recently, H\`{a}, Nguyen, Trung, and Trung proved that the depth function of a monomial ideal can be any numerical function that is asymptotically constant, \cite{HNTT}. It is known that $\depth(R/I^t)$ is not necessarily a nonincreasing function in the case $I$ is a squarefree monomial ideal, see \cite[Theorem~13]{Kaiser et. al.}. In general, for a monomial ideal $I$, lower bounds of $\depth(R/I)$ have been studied by various authors, see  \cite{DHS, DaoSchweig1, DaoSchweig2, FHM-ini, FHM-hyperforest, FHM-square, FM-lower, Lin-Mantero, squarefreePopescu}. However, the exact value for $\depth(R/I)$ is not known for arbitrary monomial ideals. 
For an arbitrary homogenous ideal $I$,  one way to determine a lower bound of $\depth(R/I)$ is to pass to an initial ideal using the fact that $\depth(R/I) \ge \depth(R/\ini(I))$, see \cite[Theorem~3.3.4]{HerzogHibi}. Moreover, equality holds when $\ini(I)$ is squarefree, \cite[Corollary~2.7]{Conca}. Exploiting this fact, Fouli, H\`{a}, and Morey introduced the notion of \textit{initially regular sequences} on $R/I$, \cite{FHM-ini}.  We recall the definition below.

\begin{definition}\cite[Definition~1.1]{FHM-ini}
    Let $R$ be a polynomial ring over a field and $I$ a proper ideal of $R$. Let $>$ be a fixed term order and set $I_1 = \ini(I)$. We say that a sequence of nonconstant polynomials $f_1,\ldots, f_q$ is an \textit{initially regular sequence} on $R/I$ if for each $i$, $1\le i\le q$, $f_i$ is a regular element of $R/I_i$, where $I_i = \ini(I_{i-1},f_{i-1})$.
\end{definition}
 These initially regular sequences play a similar role as regular sequences and give an effective lower bound for the depth of $R/I$ when $I$ is a homogeneous ideal, see \cite[Proposition~2.3]{FHM-ini}. In many instances, their length accurately computes the depth of $R/I$. Even though apriori the definition of initially regular sequences appears rather cumbersome, Fouli, H\`{a}, and Morey show that if $I$ is a monomial ideal one can use the underlying combinatorics to determine such initially regular sequences that compute the depth of $R/I$ and in some cases the depth of $R/I^t$ for $t\ge 1$, see \cite{FHM-ini, FHM-hyperforest, FHM-square}.
 
  Recall that for a graph $G$ on $n$ vertices, say $x_1, \ldots, x_n$, we let $R=k[x_1, \ldots, x_n]$ denote the polynomial ring over a field $k$, where $x_1, \ldots, x_n$ are now variables by abuse of notation. The edge ideal of the graph $G$, denoted $I(G)$, is the monomial ideal in $R$ generated by the monomials of the form $x_ix_j$, where $\left\{x_i,x_j\right\}$ is an edge of graph $G$. We refer the reader to \cite{graphtheory} for more information on graph theory.  One case where the initially regular sequences established in \cite[Theorem~3.11]{FHM-ini} do not capture the depth of $R/I$ is the case of the edge ideal of a cycle of length $3n+2$, where $n \in\mathbb{N}$, see Remark~\ref{sequences}.

One of our goals is to determine an initially regular sequence of length $n+1$ on $R/I$, where $I$ is the edge ideal of a cycle of length $3n+2$; this is accomplished in Theorem~\ref{C_{3n+2}}. The idea of our approach comes from the fact that for a nonzero module $M$ of a Noetherian ring $R$, the set of zero divisors of $M$ is the union of all associated primes of $M$. We know that an associated prime of $R/I$ corresponds to a vertex cover for the graph of the ideal $I$, when $I$ is the edge ideal of a graph or hypergraph. 

Recall that for an ideal $I$ in a Noetherian ring $R$, we define the set of associated primes of the ideal $I$ as follows
$$\Ass(R/I) = \left\{P\subseteq R \mid P \text{ is prime and }P= (I : c) \text{ for some }c\in R\right\}.$$
The set $\Ass(R/I^t), t\ge 1$ has been studied extensively in the case $I$ is the edge ideal of a graph, see \cite{Francisco et. al.}, \cite{asspowerofsquarefree}, \cite{hien2019saturation}, \cite{Martinez}. In \cite{eardecom} Lam and Trung introduced the notion of \textit{ear decompositions} to describe explicitly $\Ass(R/I^t)$ for any edge ideal $I$ and any $t\ge 1$. However, there is still very little known about $\Ass(R/I^t)$, when $I$ is an arbitrary monomial ideal, even for $t=1$. 

The paper is organized as follows. In Section~\ref{associated primes} for a monomial ideal $I$, we describe the minimal associated primes of $R/\ini(I,f)$, where $f$ is a linear sum, in terms of the minimal associated primes of $R/I$, see  Lemma~\ref{min primes leaf pair} and Proposition~\ref{mainprop1}. Moreover, we show that the embedded associated primes of $R/I$ can be described by taking the union of minimal associated primes of $R/I$ and some appropriate variables coming from the set of \textit{star neighbors} in the ideal $I$, see Definition~\ref{star neighbors} and Theorem~\ref{mainthm}. This description allows us to show that certain classes of monomial ideals have no embedded associated primes, Corollaries~\ref{no embedded},~\ref{bracket powers}. Using this description we can determine another type of regular element on monomial ideals, Corollary~\ref{new regular}. In Section~\ref{depth section} we establish an initially regular sequence that realizes the depth of $R/I(C_{3n+2})$ for any $n\ge 1$, Theorem~\ref{C_{3n+2}}. Furthermore, we accurately compute the depth of certain unicyclic graphs, Theorem~\ref{G_nm}.

\section{Associated primes of monomial ideals}\label{associated primes}
In this section, we describe the associated primes of monomial ideals in a polynomial ring. Before we proceed we fix some notation. Let $R$ be a polynomial ring and $I$ a monomial ideal. Let $\mathcal{G}(I)$ denote the set of monomial generators of $I$.  A variable $a$ is a {\it{leaf}} for $I$ if there exists a unique monomial $M \in \mathcal{G}(I)$ such that $a\mid M$.  For any variable $x$ and any monomial $M$ in $R$, we let $d_x(M)$ denote the degree of $x$ in $M$.  Moreover, $d_x(I)=\max\{d_{x}(M)\mid  M\in \mathcal{G}(I)\}$. 

First, we will examine the minimal associated primes of the initial ideal $\ini(I, f)$, where $f$ is a binomial, for an arbitrary monomial ideal $I$.  

\begin{lemma}\label{leaf}
Let $R$ be a polynomial ring and $I$ a monomial ideal in $R$. Let $a, b$ be variables in $R$ and suppose that $a$ is a leaf in $I$ and $ab\mid M$ for some $M\in \mathcal{G}(I)$. Let $I_1 = {\ini(I,a+b)}$, where $>$ is an order such that $a>b$. If $Q\in\Min(R/I_1)$, then $Q = (P,a)$ for some $P\in \Min(R/I)$.
\end{lemma}
\begin{proof} Suppose that $\mathcal{G}(I) = \left\{M_1,\ldots, M_p\right\}$. Without loss of generality, we may assume that $M_1 = a^{r_1}b^{r_2}x$, for some monomial $x$ with $r_1 = d_a(M_1)$ and $r_2 = d_b(M_2)$.

 Using \cite[Lemma~2.4]{FHM-ini} we see that $I_1=\langle a, \widehat{M} \mid M\in \mathcal{G}(I)\rangle$, where $\widehat{M} = b^{d_{a}(M)}\dfrac{M}{a^{d_{a}(M)}}$. Since $a$ is a leaf in $I$, then $M_1$ is the only monomial that $a$ divides and therefore,  $\widehat{M_1} = b^{r_1+r_2}x$ and $\widehat{M_i}=M_{i}$ for all $i>1$.  Therefore, $$I_1=\langle a, b^{r_1+r_2}x, M_2, \ldots, M_p\rangle.$$

Let $Q$ be a minimal associated prime of $R/I_1$. Then we may write $Q = (a,\,x_1,\,\ldots,\,x_r)$ for some distinct variables $x_1,\ldots,x_r$. Consider $P = (x_1,\,\ldots,\,x_r)$. We claim that $P\in\Min(R/I)$. Indeed, since $Q \in \Min(R/I_1)$, then for each $2\le k \le p$, there exists a variable $x_i$ such that $x_i \mid M_k$ with $1\le i\le r$ and  there exists a variable $x_j$ such that $x_j\mid \widehat{M_1}$ with $1\le j\le r$. Thus, $x_j \mid M_1$ as well and hence $I\subseteq P$. Suppose that $P$ is not minimal. Then there exists $P'\in\Min(R/I)$ such that $P'\subsetneq P$. Without loss of generality, we may assume that $x_1 \notin P'$. Since $P'\in\Min(R/I)$, it follows that there exists $x_{\ell}\in P'$ such that $x_{\ell} \mid M_1$ for some $2\le \ell \le r$. Since $x_{\ell}\ne a$, then $x_{\ell}\mid b^{r_2}x$ and hence $x_{\ell}\mid b^{r_1+r_2}x$. Moreover, since $I\subseteq P'$, then $M_2,\ldots, M_n \in P'$. Therefore, $I_1\subseteq (P', a)$. It is clear that $(P',a) \subsetneq Q$, which is contrary to the minimality of $Q$. Therefore, $P$ is minimal, completing the proof.
\end{proof}

Next, we recall the definition of a leaf pair from \cite{FHM-ini}. 

\begin{definition}\cite[Definition~4.10]{FHM-ini}
    Let $I$ be a monomial ideal in a polynomial ring R and let $x, y$ be two leaves in $I$. If $M_1 \ne M_2$ are the unique monomial generators in $I$ such that $x\mid M_1$ and $y\mid M_2$, and there exist monomials $z, w\in R$ with $\gcd(z,w)=1$ such that $x\nmid z$, $z\mid M_1$, $y\nmid w, w\mid M_2$, and $zw\in I$, then $x$ and $y$ is called a leaf pair. 
\end{definition}

If $I$ is the edge ideal of a graph, then a leaf pair is a pair of leaves that are distance $3$ apart. The following lemma establishes a similar result as in Lemma~\ref{leaf} when $a,b$ is a leaf pair.
\begin{lemma}\label{leafpair}
Let $R$  be a polynomial ring and $I$ a monomial ideal in $R$. Suppose $a,\,b$ is a leaf pair and let $I_1 = \ini(I, a+b)$ with respect to a term order $>$. If $Q\in\Min(R/I_1)$, then either $Q = (P,a)$ or $Q = (P,b)$ for some $P\in\Min(R/I)$.
\end{lemma}
\begin{proof}
Let $\mathcal{G}(I) = \left\{M_1,\ldots,M_p\right\}$. Without loss of generality we may assume that $a>b$, and $M_1$, $M_2$ are the unique monomials that are divisible by $a, b$, respectively. Let $r_1 = d_a(M_1)$ and $r_2 = d_b(M_2)$. Then we can write $M_1 = a^{r_1}zx$ and $M_2 = b^{r_2}wy$, with $z, w,x,y$ monomials, $a\nmid z$, $b\nmid w$, and $zw\in \mathcal{G}(I)$. Without loss of generality, assume $M_3 = zw$. By \cite[Lemma~2.4]{FHM-ini}, we have  $I_1 = (a, b^{r_1}zx, b^{r_2}wy, zw, M_4,\ldots, M_p)$. Let $Q\in \Min(R/I_1)$ and notice that $a\in Q$. We have two cases to consider. 

First, suppose $b\in Q$. Then we can write $Q=(a, b, x_1, \ldots, x_r)$ with $x_1,\ldots, x_r$ distinct variables. Notice that since  $b\nmid M_i$ for all $i\ge 3$, then $(zw, M_4,\ldots, M_p) \subseteq (x_1, \ldots, x_r)$.  Moreover, there exists $x_i$ with $1\le i \le r$ such that $x_i \mid zw$. Then either $x_i \mid z$ or $x_i \mid w$. 

Suppose first that $x_i \mid z$. Let $P=(b, x_1, \ldots, x_r)$.  As $x_i \mid z$, then $x_j \nmid wy$ for all $j$, since otherwise $b$ would be redundant in $Q$, contradicting the minimality of $Q$. Notice that  $x_i \mid a^{r_1}zx$ and thus $I \subseteq P$. We claim that $P\in \Min(R/I)$. Suppose instead that there exists a prime $P'\subsetneq P$ such that $I\subseteq P'$.  If $b\not\in P'$, then there exists $x_j \in P'$ such that $x_j \mid wy$, a contradiction since $x_j \in P$. Therefore, without loss of generality, we may assume that $x_1\notin P'$.  Hence $(zw, M_4,\ldots, M_p) \subseteq (x_2, \ldots, x_r)$ and since $b\in P'$ , $I\subseteq P'$ and $x_1\neq a$,  then $I_1 \subseteq (P',a) \subsetneq (P,a)=Q$, contradicting the minimality of $Q$. Hence $P\in \Min(R/I)$ and  $Q=(P,a)$. 

Next suppose $x_i \mid w$. Then $x_j \nmid zx$ for all $j$, since $b$ would be redundant in $Q$. Let $P=(a, x_1, \ldots, x_r)$. We claim that $P\in \Min(R/I)$. Suppose instead that there exists a prime $P'$ with $I\subseteq P' \subsetneq P$. Notice that $a \in P'$, since $x_j \nmid zx$ for all $j$. Therefore, without loss of generality, we may assume that $x_1\not\in P'$. Hence $(zw, M_4,\ldots, M_p) \subseteq (x_2, \ldots, x_r)$ and thus $I_1 \subseteq (P', b)\subsetneq (P,b)=Q$, contradicting the minimality of $Q$. Therefore, $P\in \Min(R/I)$ and $Q=(P, b)$ in this case.

Finally, it remains to consider the case $b\not \in Q$. We may write $Q=(a, y_1, \ldots, y_t)$ for some variables $y_1, \ldots, y_t$. Let $P=(y_1, \ldots, y_t)$ and notice that since $b\not\in Q$, then there exists $y_j$ with $1\le j \le r$ such that $y_j \mid zx$. Hence $y_j \mid a^{r_1}zx$ and thus $I\subseteq P$. As before, we claim that $P \in \Min(R/I)$. Indeed, if there exists a prime $P' \subsetneq P$ with $I\subseteq P$, then without loss of generality we may assume that $y_1 \not\in P'$. Then since $a\not\in P'$, then $zx\in P'$ and thus $I_1 \subseteq (P',a) \subsetneq (P,a)=Q$, a contradiction to the minimality of $Q$. Thus $P\in \Min(R/I)$ and $Q=(P,a)$ in this case. 
\end{proof}

The converse of Lemma~ \ref{leafpair} is also true in the case of edge ideals of graphs. Recall that if $I$ is the edge ideal of a graph and $x$ is a vertex of the graph, then $N(x)$ is the neighbors of $x$, that is $N(x)=\{y\in R\mid y \mbox{ is a variable, } xy\in \mathcal{G}(I)\}$. 
 
\begin{lemma}\label{min primes leaf pair}
Let $R$ be a polynomial ring, $I$ the edge ideal of a graph, and let $P\in\Min(R/I)$. Suppose that $a,\,b$ is a leaf pair in $I$ and let $I_1=\ini(I,a+b)$ with respect to an order such that $a > b$.  If $a\not\in P$, then $(P,a)\in \Min(R/I_1)$  and if $a\in P$, then $(P,b)\in \Min(R/I_1)$.
\end{lemma}

\begin{proof}
Since $a,\,b$ is a leaf pair then there exist $M_1,\,M_2, M_3\in \mathcal{G}(I)$ such that $M_1 = ax$, $M_2 = by$ and $M_3 = xy$, where $x,\,y$ are distinct variables of $R$. Let $\mathcal{G}(I)=\{M_1, \ldots, M_p\}$ and notice that by construction $a, b \nmid M_i$ for all $i\ge 3$. By \cite[Lemma~2.4]{FHM-ini}, we have $$I_1 = \ini(I,\,a+b) = \left(a,\,bx,\,M_i\mid 2\le i\le p\right).$$
Let $P\in\Min(R/I)$. Then $P = (x_{1},\ldots, x_{s})$ for some distinct variables $x_{1},\ldots, x_{s}$. We remark that since $M_3 = xy\in\mathcal{G}(I)$, then $a$ and $b$ can not both be in $P$. Moreover, note that every associated prime of $R/I_1$ contains the variable $a$ by the construction of $I_1$. We have two cases to consider. 

First, suppose that $a\notin P$. We claim that $Q=(P,a)\in \Min(R/I_1)$. Since $ax\in I$ and $a\not\in P$, then $x\in P$. Hence $I_1 \subseteq (P,a)=Q$. Suppose $Q\not\in \Min(R/I_1)$. Then there exists a prime $Q'$ such that $I_1\subseteq Q'\subsetneq Q$. Therefore, $a\in Q'$ and thus without loss of generality we may assume that $x_1\not\in Q'$. If $b\notin P$, then $y\in P$. Moreover, we have that $b\notin Q'$ and thus $x, y\in Q'$. Hence $x_{1}\notin \left\{a, b, x, y\right\}$. Now, since $x_{1}\notin Q'$, then $N(x_{1})\subseteq Q'$, and thus $N(x_{1})\subseteq P$. Now, since $x_{1}\notin \left\{a, b, x, y\right\}$, then for all $v\in N(x_{1})$, there exists $M_i\in I_1$ with $i\ge 4$ such that $M_i = vx_{1}$. Notice that $M_i\in I$ as well since $i\ge 4$. Hence, we have that $x_{1}\in P$ and $N(x_{1})\subseteq P$, which is a contradiction to the minimality of $P$. Therefore, $Q = (P,a)\in \Min(R/I_1)$. If $b\in P$, then we see that $I_1\subseteq (P,a)$ and $y\notin P$. Since $y\notin P$, then $y\notin Q'$, and hence $x, b\in Q'$ and $x_{1}\ne x, b$. Again, $x_{1}\ne y$ since $y\notin P$ and $x_{1}\in P$. Therefore, $x_{1}\notin\left\{a, b, x, y\right\}$. By a similar argument as above, we arrive at a contradiction. Hence, $Q = (P, a) \in \Min(R/I_1)$.

Finally, suppose $a\in P$. Then we claim that $Q=(P,b) \in \Min(R/I_1)$.  Since $a\in P$, then $b\not\in P$ and thus $x\notin P$ and  $y\in P$. Clearly, we have $I_1 \subseteq (P, b)$. Suppose there exists a prime  $Q$ such that $I_1\subseteq Q\subsetneq (P, b)$. Then without loss of generality, we may assume that  $x_{1}\notin Q$. Since $x\notin P$, then $x\notin Q$, and hence $b, y\in Q$. So $x_{1}\ne b, y$. Moreover,  $x_{1}\ne x$ since $x_{1}\in P$ and $x\notin P$. Thus, $x_{1}\notin\left\{a, b, x, y\right\}$. Using a similar argument as above, we have a contradiction again. Therefore, $Q = (P, b)\in \Min(R/I_1)$. 
\end{proof}

Our next goal is to investigate the minimal associated primes of $\ini(I, f)$, where $f$ is a trinomial. We recall a lemma from \cite{FHM-square} which gives a description of the initial ideal $\ini(I,f)$, when $f$ is a trinomial. 

\begin{lemma}\cite[Lemma~4.1]{FHM-square}\label{lemmatrinomial}
    Let $I$ be a monomial ideal. Let $a, b, c$ be variables satisfying the following condition:
\begin{enumerate}
\item[$($a$)$]  $d_{a}(I), d_b(I), d_c(I) \le 1$ and 
\item[$($b$)$] If $M\in\mathcal{G}(I)$  and  $a \mid M$, then either  $b\mid M$ or  $c \mid M$.
\end{enumerate}
 Furthermore, assume that $bc\nmid M$ for any $M\in\mathcal{G}(I)$. If $>$ is a term order such that $a > b > c$, then
$$\ini(I, a+b+c) = \langle a, \widehat{M}, \lcm(X,M')c^2 \mid M, bX, acM'\in \mathcal{G}(I)\rangle.$$
\end{lemma}

The next proposition establishes a connection between the minimal associated primes of an ideal and its initial ideal with a trinomial.

\begin{proposition}\label{mainprop1}
Let $R$ be a polynomial ring, $I$ a monomial ideal in $R$, and $a,b,c $ distinct variables in $R$ such that $\mathcal{G}(I)= \left\{ab, ac, M_1,\ldots, M_p\right\}$ with  $a \nmid M_i$ for $1 \le i \le p$. Moreover, suppose that $d_b(I) = d_c(I) = 1$. Let $I_1 = \ini(I, a+b+c)$, where $>$ is an order such that $a > b > c$. If $Q \in\Min(R/I_1)$, then there exists $P\in\Min(R/I)$ such that either $Q = (P,a)$ if $c\in P$ or $Q = (P,b)$ if $c\not\in P$. 
\end{proposition}
\begin{proof}
Let $Q\in\Min(R/I_1)$. Without loss of generality, let $M_1,\ldots,M_r$ be the monomials such that $b\mid M_i$ for $1 \le i\le r$, that is $M_i = bM'_i$ for some monomial $M'_i$ and $b\nmid M_j$ for any $j>r$. If there do not exist such monomials, then $r=0$ and the argument still goes through. By Lemma~\ref{lemmatrinomial}, we have that
$$I_1 = (a, b^2, bc, M_1,\ldots, M_p, M'_1c^2, \ldots, M'_rc^2).$$
Notice that from the description of $I_1$, one can see that $a, b\in Q$ for any $Q\in\Min(R/I_1)$.  Hence we may write $Q = (a,b,y_1,\ldots,y_t)$ for some distinct variables $y_1, \ldots, y_t$. We consider two possible cases whether $c\in Q$ or $c\not\in Q$. 

First, assume $c\notin Q$. Let $P = (a, y_1,\ldots, y_t)$. We will show that $P\in\Min(R/I)$ and thus $Q=(P,b)$ in this case. Since  $a\nmid M_i$ for $1\le i\le p$, then $a\nmid M_i'c^2$ for $1\le i\le r$ and hence it follows that $M_1,\ldots, M_p, M'_1c^2,\ldots, M'_rc^2 \in (b, y_1,\ldots, y_t)$. Notice that since $d_b(I)=1$, then for each $1\le i\le r$, $b\nmid M'_ic^2$. Hence for each $1\le i \le r$ there exists $1\le j\le t$ such that $y_j \mid M'_ic^2$. Since $y_j \ne c$, then $y_j\mid M'_i$ and hence $y_j\mid M_i$. Thus $M_1, \ldots, M_r \in (y_1, \ldots, y_t)$. Moreover, since $b\nmid M_i$ for each $r+1 \le i\le p$, then $M_{r+1}, \ldots, M_p\in (y_1, \ldots, y_t)$. Therefore, $I\subseteq P$. Suppose now that $P$ is not minimal, that is there exists $P'\in\Min(R/I)$ such that $P'\subsetneq P$. Then, without loss of generality, we may assume that $y_1\notin P'$. Notice that $y_1 \ne a$ and hence  $I\subseteq P'\subseteq (a, y_2,\ldots, y_t)$. For $1\le i\le r$, there exists $j$ with $2\le j \le t$ such that $y_j\mid M_i = bM'_i$. Since $y_j\ne b$, then $y_j\mid M'_i$ and thus, $y_j\mid M'_ic^2$. For $r+1\le i\le p$, there exists $j$ with $2\le j\le t$ such that $y_j\mid M_i$. Therefore,  $M_1,\ldots, M_p, M'_1c^2,\ldots, M'_rc^2 \in (y_2,\ldots, y_t)$. Hence, $I_1\subseteq (a, b, y_2,\ldots, y_t)\subsetneq Q$, a contradiction to the minimality of $Q$. Therefore, $P\in\Min(R/I)$ and $Q = (P, b)$.

 For the remaining case, assume $c\in Q$. Then we may write $Q = (a, b, c, y_1,\ldots, y_t)$. Let $P = (b, c, y_1,\ldots, y_t)$. It suffices to show that $P\in\Min(R/I)$. Clearly, $b\mid ab$ and $c\mid ac$. For $1\le i\le r$, by our assumption, $b\mid M_i$. For each $M_i$ with $r+1\le i\le p$ such that $c\nmid M_i$, there exists $y_j\mid M_i$ for some $1\le j\le t$. Thus, $I\subseteq P$. Suppose there exists $P'\in\Min(R/I)$ such that $P'\subsetneq P$. Notice that since $a\notin P$, then $a\notin P'$, and hence $b, c\in P'$. Then without loss of generality, we may assume that $y_1\notin P'$. We see that $I\subseteq P'\subseteq (b, c, y_2,\ldots, y_t)$. Clearly, one can see that $M_1,\ldots, M_r, M'_1c^2,\ldots, M'_rc^2 \in (b, c)$. Moreover, for each $M_i$ with $r+1\le i\le t$ such that $c\nmid M_i$, there exists $y_j\mid M_i$ for some $1\le j\le t$. Therefore, $I_1\subseteq (a, b, c, y_2,\ldots, y_t)\subsetneq Q$, which contradicts to the minimality of $Q$. Therefore, $P\in\Min(R/I)$ and $Q = (P, a)$. \end{proof}

Our next goal is to describe the embedded associated primes of an arbitrary monomial ideal $I$. We will use the concept of polarization, a convenient tool that will help us to work with the minimal associated primes in the polarized ring. We recall the definition of polarization below.
\begin{definition}\cite[Definition~2.1]{Faradi-polarization}\label{pola-def}
    Let $R = k[x_1,\ldots,x_n]$ be a polynomial ring over a field $k$. Suppose $M = x_1^{a_1}\cdots x_n^{a_n}$ is a monomial in $R$ with $a_i \ge 1$ for all $1\le i\le n$. We define the \textit{polarization} of $M$ to be the squarefree monomial 
    $$\mathcal{P}(M) = x_{1,1}x_{1,2}\cdots x_{1,a_1}\cdots x_{n,1}\cdots x_{n,a_n}$$
    in the polynomial ring $R^{\pol} = k[x_{i,j} \mid 1\le i \le n, 1\le j \le a_i]$.

    If $I$ is an ideal of $R$ generated by a set of monomials $M_1,\ldots,M_q$, then the polarization of $I$ is defined as:
$$I^{\pol} = (\mathcal{P}(M_1),\ldots, \mathcal{P}(M_q))$$
which is a squarefree monomial ideal in the polynomial ring $R^{\pol}$.
\end{definition}
The following lemma establishes the connection between the associated primes of an ideal and its polarization. 
\begin{lemma}\cite[Corollary~2.6]{Faradi-polarization}\label{polarization primes}
Let $I$ be a monomial ideal in a polynomial ring $R = k[x_1,\ldots,x_n]$, and let $I^{\pol}$ be its polarization in $R^{\pol}= k[x_{i,j}]$. Then $(x_{i_1},\ldots, x_{i_r})\in\Ass(R/I)$ if and only if $(x_{i_1,c_1},\ldots, x_{i_r,c_r})\in\Ass(R^{\pol}/I^{\pol})$ for some positive integers $c_1,\ldots, c_r$.
\end{lemma}

In the setting of Lemma~\ref{polarization primes} for an associated prime $$W=(x_{i_1,c_1},\ldots, x_{i_r,c_r})\in\Ass(R^{\pol}/I^{\pol}),$$ we denote $W^{\depol}$ the depolarization of $W$, that is $W^{\depol}=(x_{i_1},\ldots, x_{i_r})\in\Ass(R/I)$.  

 We introduce a special set of neighbors that we will use for the next theorem. 

\begin{definition}\label{star neighbors}
  Let $I$ be a monomial ideal in a polynomial ring. For a variable $w$, we define the set
$$N^*(w) = \left\{z \in R \mid z \mbox{ is a variable, } z\neq w,  \exists M\in\mathcal{G}(I) \mbox{ with } zw\mid M \mbox{ and }d_w(M) < d_w(I)\right\}.$$

\end{definition}
We give an example to illustrate these special neighbor sets. 
\begin{example}
Let $R=k[a,b,c,d]$ be a polynomial ring over a field $k$ and let $I=(a^2bc, ad, b^3cd)$ be a monomial ideal in $R$. Notice that $N^{*}(w)\neq \emptyset$ for $w=a,b$ since $d_{a}(ad)<d_a(a^2bc)$ and $d_b(a^2bc) <d_b(b^3cd)$. Moreover, $N^{*}(a)=\{d\}$ and $N^{*}(b)=\{a,c\}$. 
\end{example}

We are now ready to prove the main theorem of this section. We show that for an arbitrary monomial ideal $I$, any embedded associated prime of $R/I$ can be expressed as the sum of a minimal associated prime of $R/I$ with an ideal of some additional appropriate variables.

The following theorem is the main theorem of this section. 
\begin{theorem}\label{mainthm}
Let $R$ be a polynomial ring over a field and  $I$ be a monomial ideal in $R$. Let $w_1,\ldots, w_n$ be the distinct variables in $R$ such that $N^{*}(w_i)\neq \emptyset$. 
Let $Q\in\Ass(R/I)$ be an embedded associated prime. Then $Q = (Q', z_1, z_2,\ldots, z_t)$ for some $Q'\in\Min(R/I)$ and $z_j\in N^*(w_i)$ for some $1\le i\le n$ and $1\le j\le t$.
\end{theorem}

\begin{proof}
Suppose that $Q$ is an embedded associated prime ideal of $R/I$. Then there exists $Q'\in\Min(R/I)$ such that $Q'\subsetneq Q$. 
We claim that 
$$Q = (Q', z_1,\ldots, z_t \mid z_j \in N^*(w_i), 1\le i \le r, 1\le j\le t).$$

Since $Q'\subsetneq Q$, there exists a variable $v\in Q\setminus Q'$.
Suppose that $v\notin N^*(w_i)$ for all $1\le i\le r$. Let $$\mathcal{A} = \left\{N\in\mathcal{G}(I) \mid v\mid N\right\}.$$ Since $v\notin Q'$ and $Q'\in\Min(R/I)$, then for each $N\in\mathcal{A}$, there exists a variable $t_N\in Q'$ such that $vt_N \mid N$. Let $N = v^{m_N}t_N^{e_N}N'$ for some monomial $N'$, $m_N, e_N \ge 0$, and $v,t_N\nmid N'$.  By Lemma~\ref{polarization primes} we have that $Q = W^{\text{depol}}$ for some $W\in\Min(R^{\pol}/I^{\pol})$.  Let $t_{N, 1},\ldots, t_{N, e_N}$ be the polarizing variables of $t_N$ with $t_{N, 1} = t_N$ and $v_{1},\ldots, v_{m_N}$ be polarizing variables of $v$ with $v_1 = v$ in $R^{\pol}$.

We have two cases to consider, whether $v=w_i$ for some $i$ or if $v\neq w_i$ for all $i$.

First suppose that $v=w_{i_0}$ for some $1\le i_0 \le n$. By our assumption $w_{i_0}\not\in N^{*}(w_i)$ for all $i$.  Hence for each $N\in\mathcal{A}$, by our assumption, we have that $t_N\ne w_i$ for all $1\le i\le n$.  Since $v\in Q$, then there exists some $v_q \in W$ for some $1 \le q \le m_N$. If there exists some $L\in\mathcal{G}(I)$ with $0 < d_{t_N}(L) < e_N$ or $d_{t_N}(L) > e_N$ then $t_N = w_u$ for some $1 \le u \le r$, in other words, $v\in N^{*}(w_u)$, a contradiction. Hence, $d_{t_N}(I) = e_N$ and for each $M\in \mathcal{G}(I)$ such that $t_N \mid M$, then $d_{t_N}(M)=e_N$. Therefore, for each $N\in\mathcal{A}$, there exists $t_N\in Q'$ such that $d_{t_N}(N) = d_{t_N}(I)$. Since $t_N\in Q'$, then $t_N\in Q$. Hence $t_{N,j}\in W$ for some $1\le j\le e_N$. Thus, $v_q$ and $t_{N,j}$ are both in $W$, a contradiction to the minimality of $W$.

Finally, suppose that $v \ne w_i$ for all $1\le i\le n$. Since $v\ne w_i$, then $d_v(I) = d_v(N)$ for all $N\in\mathcal{A}$, that is $m_N = d_v(I) = m$ for some $m\in\mathbb{N}$. Since $v\in Q$, then there exists $v_q\in W$ for some $1 \le q \le m$. Notice that if $d_{t_N}(N) < d_{t_N}(I)$, then $t_N = w_{i_0}$ for some $1 \le i_0\le r$ and hence $v\in N^*(w_{i_0})$, which is a contradiction. Therefore, $e_N = d_{t_N}(I)$ for all $N\in\mathcal{A}$. Since $t_N \in Q'$, then $t_N\in Q$. Hence $t_{N,j} \in W$ for some $1 \le j\le e_N$ and  $v_q$ and $t_{N,j}$ are both in $W$, a contradiction to the minimality of $W$ as before.   \end{proof}

The next example illustrates the statement of Theorem~\ref{mainthm} and how the embedded associated primes arise from minimal associated primes.

\begin{example}
Consider the ideal $I = (a^3bc, a^2d,b^2c,ce^2,de,c^2f,eg)$ in the polynomial ring $R=\mathbb{Q}[a,b,c,d,e,f,g]$. Notice that $N^*(w)\neq \emptyset$ for $w\in \{a,b,c,e\}$.
Indeed, $N^*(a) = \left\{d\right\}$, $N^*(b) = \left\{a,c\right\}$, $N^*(c) = \left\{a, b, e\right\}$, $N^*(e) = \left\{d, g\right\}$. Using Macaulay $2$ \cite{m2},  the list of associated primes of $I$ is:
\begin{align*}
    P_1 &= (a, c, e), P_2 = (c, d, e), P_3 = (c, d, g),  P_4 = (a, b, e, f), P_5 = (b, d, e, f),\\
    Q_1 &= (a, b, c, e), Q_2 = (b, c, d, e), Q_3 = (a, b, c, d, e), Q_4 = (a, b, d, e, f),\\
    Q_5 &= (b, c, d, e, g), Q_6 = (b, d, e, f, g), Q_7 = (a, b, c, d, e, g), Q_8 = (a, b, d, e, f, g).
\end{align*}
Notice that $P_1, \ldots,  P_5 \in \Min(R/I)$ and $Q_1, \ldots, Q_8$ are embedded associated primes of $R/I$. Notice that $Q_1 = (P_1, b)$, $Q_2 = (P_2 , b)$, $Q_3 = (P_1 , b, d) = (P_2, a, b)$, $Q_4 = (P_4, d) = (P_5, a)$, $Q_5 = (P_2, b, g) = (P_3, b, e)$, $Q_6 = (P_5, g)$, $Q_7 = (P_1, b, d, g) = (P_2, a, b, g) = (P_3, a, b, e)$, $Q_8 = (P_4, d, g) = (P_5, a, g)$.

In particular, notice that $Q_7=(P_1,b,d,g)$ with $b\in N^{*}(w_3)$ and $d,g\in N^{*}(w_4)$, showing that the variables in an embedded associated prime can come from one or more of $N^{*}(w_i)$.
\end{example}

For any monomial $M$, let $\widetilde{M}$ denote the squarefree part of $M$. That is if $M=x_1^{a_1}\cdots x_n^{a_n}$, then $\widetilde{M}=x_1\cdots x_n$.

\begin{corollary}\label{no embedded}
Let $I$ be a monomial ideal in a polynomial ring $R$ and suppose that for every variable $x\in R$ we have $d_{x}(M)=d_x(I)$ for every monomial $M \in \mathcal{G}(I)$ such that $x\mid M$. Then $R/I$ has no embedded primes. Furthermore, $\Min (R/I)=\Min(R/J)$, where $J=\langle \widetilde{M}\mid M\in \mathcal{G}(I)\rangle$.
\end{corollary}

\begin{proof}
This follows immediately from Theorem~\ref{mainthm} by noticing that $N^*(x_i)=\emptyset$ for all $i$.
\end{proof}

The following example illustrates Corollary~\ref{no embedded}

\begin{example}
Let $R=k[x_1,\ldots, x_4]$ and let $I=(x_1^3x_2,x_2x_3^2, x_3^2x_4^4,x_1^3x_4^4)$. Then $I$ has the same associated primes as $J=(x_1x_2,x_2x_3,x_3x_4,x_1x_4)$, that is 
$$\Ass(R/I)=\Min(R/I)=\Ass(R/J)=\Min(R/J)=\{(x_1,x_3), (x_2,x_4)\}.$$    
\end{example}

Recall that for an ideal $I$ in a Noetherian ring $R$ we have $I^{[n]}=\{f^n \mid f\in I\}$. The following shows that these types of ideals have no embedded associated primes when $I$ is squarefree or satisfies the assumptions of Corollary~\ref{no embedded}.

\begin{corollary}\label{bracket powers}
Let $I$ be either a squarefree monomial ideal or a monomial ideal that satisfies the assumptions of Corollary~\ref{no embedded} in a polynomial ring $R$. Then for all $n\ge 1$, $I^{[n]}$ has no embedded associated primes. 
\end{corollary}

Finally, we determine another type of regular element on a monomial ideal. This is a generalization of \cite[Lemma~3.9]{FHM-ini}, where the degree of all but one variable was required to be $1$.

\begin{corollary}\label{new regular}
Let $I$ be a monomial ideal in a polynomial ring $R$ and suppose that for every variable $x\in R$ we have $d_{x}(M)=d_x(I)$ for every monomial $M \in \mathcal{G}(I)$ such that $x\mid M$.
Let $b_0, b_1, \ldots, b_t$ be distinct variables in $R$ such that if $b_0\mid M$ for some monomial $M\in \mathcal{G}(I)$, then $b_i \mid M$ for some $1\le i\le t$. If $f=b_0+\ldots +b_t$, then $f$ is regular on $R/I$.
\end{corollary}

\begin{proof}
 First notice that by Corollary~\ref{no embedded} we know $R/I$ has no embedded associated primes and in fact $\Min(R/I)=\Min(R/J)$, where $J=\langle \widetilde{M} \mid M \in \mathcal{G}(I)\rangle$. By \cite[Lemma~3.9]{FHM-ini} $f$ is regular on $R/J$ and therefore, $f$ is regular on $R/I$.   
\end{proof}

We close with an example of the new type of regular elements obtained in Corollary~\ref{new regular}.

\begin{example}
Let $R=\mathbb{Q}[x_1, x_2, x_3, x_4, x_5]$ and let $I=(x_1^3x_2^4,x_1^3x_5^2, x_2^4x_3^2, x_3^2x_4^5,x_4^5x_5^2)$. Since $d_{x_i}(I)>1$ for all $i$, then the results in \cite{FHM-ini} do not apply here. However, $I$ satisfies the assumptions of Corollary~\ref{new regular}. Let $J=\langle \widetilde{M} \mid M \in \mathcal{G}(I)\rangle$ and notice that $J=(x_1x_2,x_1x_5,x_2x_3,x_3x_4,x_4x_5)$, which is the edge ideal of a pentagon. The element $f=x_1+x_2+x_5$ is regular on $R/J$, by \cite[Lemma~3.9]{FHM-ini} and since $\Ass(R/J)=\Ass(R/I)$, then $f\notin \Ass(R/I)$, or in other words, $f$ is regular on $R/I$.
\end{example}

\section{Initially regular sequences on cycles and the depth of unicyclic graphs}\label{depth section}
In this section, we will use the description of the associated primes of monomial ideals from Section~\ref{associated primes} to construct initially regular sequences on cycles. Moreover, using these sequences we will compute accurately the depth of certain unicyclic graphs. 

We first recall the following result that gives the depth of any cycle. For $n\ge3$ let $C_n$ denote a cycle on $n$ variables and let $I(C_n)$ be the edge ideal of the cycle in $R=k[x_1, \ldots, x_n]$, where $k$ is a field, that is $I(C_n)=(x_1x_2, x_2x_3, \ldots, x_{n-1}x_n, x_1x_n)$. Note that $[n]=\{1, \ldots, n\}$ for any $n\in \mathbb{N}$.

  \begin{theorem} \cite[Corollary~7.6.30]{Jacq}\label{Jacq}
Let $n\ge3$, $C_n$ be a cycle on $n$ vertices, and $I(C_n)$ be the edge ideal of the cycle in the ring $R=k[x_1, \ldots, x_n]$, where $k$ is a field. Then $$\depth(R/I(C_n)) = \lceil\dfrac{n-1}{3}\rceil.$$    
   \end{theorem}

\begin{remark}\label{sequences}
Let $n\ge 1$  and let $I_1=I(C_{3n})$ denote the edge ideal of a cycle of length $3n$ in the polynomial ring $R_1=k[x_1, \ldots, x_{3n}]$. Let $I_2=I(C_{3n+1})$ 
denote the edge ideal of a cycle of length $3n+1$ in $R_2=k[x_1, \ldots, x_{3n+1}]$ and let $I_3=I(C_{3n+2})$ 
denote the edge ideal of a cycle of length $3n+2$ in $R_3=k[x_1, \ldots, x_{3n+2}]$. According to  Theorem~\ref{Jacq} we have $\depth R_1/I_1=\depth R_2/I_2=n$, whereas $\depth R_3/I_3=n+1$. 

Let $h_1 = x_1+x_{3n} + x_2$ and $h_i=x_{3i-2}+x_{3i-3}+x_{3i-1}$ for all $2\le i\le n$ and observe that $h_1, \ldots, h_n$ is an initially regular sequence on $R_1/I_1$ that realizes the depth of $R_1/I_1$, by \cite[Theorem~3.11]{FHM-ini}.  Similarly, let $g_1 = x_1+x_{3n+1} + x_2$ and $g_i=x_{3i-2}+x_{3i-3}+x_{3i-1}$ for all $2\le i\le n$. Then $g_1, \ldots, g_n$
 is an initially regular sequence that realizes the depth of $R_2/I_2$. However, \cite[Theorem~3.11]{FHM-ini} does not provide a method to construct an initially regular sequence of length $n+1$ that would realize the depth of $R_3/I_3$. In fact, \cite[Theorem~3.11]{FHM-ini} shows that there is an initially regular sequence of length $n$ on $R_3/I_3$, that is $f_1 = x_1+x_{3n+2} + x_2$, $f_i=x_{3i-2}+x_{3i-3}+x_{3i-1}$ for $2\le i\le n$ is an initially regular sequence on $R/I$ with respect to an appropriate term order.
\end{remark}

As noted in Remark~\ref{sequences} we have initially regular sequences that realize the depth for $I(C_{3n})$ and $I(C_{3n+1})$, for any $n$. Moreover, we have an initially regular sequence of length $n$ on $R/I(C_{3n+2})$. The next theorem establishes an initially regular sequence that realizes the depth for the edge ideal of a cycle of length $3n+2$ for $n\ge 1$. 

\begin{theorem}\label{C_{3n+2}}
Let $n\ge 1$, $C_{3n+2}$ be a cycle on $3n+2$ vertices, and let $I=I(C_{3n+2})$ be the edge ideal of the cycle $C_{3n+2}$ in the ring $R=k[x_1, \ldots, x_{3n+2}]$. 
Let $f_1 = x_1+x_{3n+2} + x_2$,  $f_i=x_{3i-2}+x_{3i-3}+x_{3i-1}$ for $2\le i\le n $, and  $f_{n+1}=x_{3n} + x_{3n+1} + \sum \limits_{i=1}^{n}x_{3i-1}$.
The sequence $f_1, \ldots, f_{n+1}$
is an initially regular sequence on $R/I$ with respect to a term order such that $x_1 > x_{3n+2} > x_2$, and $x_{3i-2}>x_{3i-3}>x_{3i-1}$, for all $2\le i\le n $.
\end{theorem}
\begin{proof}

We fix some notation. For all $i\in [n]$ and $2\le j \le n$, let $a_i=x_{3i-2}$, $b_1=x_{3n+2}$, $b_{j}=x_{3j-3}$, $c_i=x_{3i-1}$. Then $f_i=a_i+b_i+c_i$ for all $i\in[n]$. By \cite[Theorem~3.11]{FHM-ini}, we have that $f_1, f_2,\ldots, f_n$ is an initially regular sequence on $R/I$ with respect to a term order such that $x_1 > x_{3n+2} > x_2$, and $x_{3i-2}>x_{3i-3}>x_{3i-1}$, for all $2\le i\le n $. It remains then to show that $f_{n+1}$ is regular on $R/I_n$. It suffices to show that $f_{n+1}\not\in \Ass(R/I_n)$.

Applying Lemma~\ref{lemmatrinomial} repeatedly, we can give a complete description of $I_n$, that is,
\begin{align*}
    I &= (x_1x_2, x_2x_3,\ldots, x_{3n}x_{3n+1}, x_{3n+1}x_{3n+2}, x_1x_{3n+2})\\
    I_1 &= \ini(I,f_1)=(x_1, x_{i}x_{i+1}, x_{3n+2}x_2, x_{3n+2}^2,x_{3n+1}x_2^2\mid 2\le i\le 3n+1)\\
           &\,\,\,\,\vdots\\
    I_n &= \ini(I_{n-1},f_n) \\
        &=(x_{3i-2}, x_{3i-1}x_{3i}, x_{3n}x_{3n+1}, x_{3n+1}x_{3n+2}, x_{3n+2}^2, x_{3j}^2, x_{3n+1}x_2^2, x_{3j-1}x_{3j+2}^2,  \\
        & \hspace{.7cm} x_{2}x_{3n+2}, x_{3j}x_{3j+2},\mid i\in[n], j\in[n-1]).\\
            \end{align*}

Let $Q \in \Min(R/I_n)$.  Using  Proposition~\ref{mainprop1} repeatedly, there exists $P\in \Min(R/I)$ such that 
$$Q = (P, a_{{i_1}},\ldots,a_{{i_k}},b_{{i_{k+1}}},\ldots, b_{{i_n}}),$$
for some $ i_j\in [n]$ such that $i_j \ne i_r$ for all $j\ne r$. We make the following observations:
\begin{itemize}
    \item If $a_{{i_j}} = x_1$, for some $i_j$, then $x_1\notin P$ and thus $x_{3n+2}\in P$. Hence either $x_{3n}\not\in P$ or $x_{3n+1}\notin P$ by the minimality of $P$.
    \item If $b_{{i_j}} = x_{3n-3}$, for some $i_j$, then $x_{3n-3}\notin P$ and thus  $x_{3n-2}\in P$. Hence either $x_{3n-1}\not\in P$ or  $x_{3n}\notin P$ by the minimality of $P$.
    \item If $a_{{i_j}} \ne x_1, b_{{i_j}} \ne x_{3n-3}$ for all $i_j$, then $x_1,x_{3n-3}\in P$. So $x_{3n+2}, x_{3n-2}\notin P$,  by Proposition~\ref{mainprop1}. Thus, $x_{3n+1}, x_{3n-1}\in P$, and hence $x_{3n}\notin P$.
\end{itemize}
Therefore, $f_{n+1}\notin Q$ for any $Q\in\Min(R/I_n)$.

Now, we let $Q$ be an embedded associated prime of $R/I_n$.
Then there exists $Q'\in \Min(R/I_n)$  and variables $z_1, \ldots, z_t$ such that  
$$Q = (Q', z_1,\ldots,z_t).$$

Notice that $x_{3n+2}, x_{3j} \in Q'$ for all $j\in[n-1]$. Following the notation of Theorem~\ref{mainthm}, we have $w_1 = x_{3n+2}$, $w_i=x_{3i}$ for $2\le i\le n$ and for $1\le j\le n$ we have $w_{n+j}=x_{3j-1}$. That is these are the $w_i$ such that $N^*(w_i)\neq \emptyset$. Consequently, we obtain $N^{*}(x_{3n+2})=\{x_{3n+1}, x_2\}$, $N^{*}(x_{3i})=\{x_{3i-1}, x_{3i+2}\}$, $N^{*}(x_2)=\{x_3, x_{3n+2}\}$, and $N^{*}(x_{3j-1})=\{x_{3j}, x_{3j-3}\}$ for all $2\le i\le n$ and $2\le j \le n$. Therefore, by Theorem~\ref{mainthm}, we have that $z_1,\ldots,z_t\in \{x_{3i-1}, x_{3n}, x_{3n+1} \mid 1\le i\le n\}$. 

We write $Q = W^{\text{depol}}$ for some $W\in\Min(R^{\pol}/I_n^{\pol})$. By Lemma~\ref{lemmatrinomial}, we see that the degrees of $x_1,\ldots, x_{3n-1}, x_{3n+2}$ are $2$, so for each $i\in\left\{1, 2,\ldots, 3n-1, 3n+2\right\}$, by abuse of notation, we let $x_i$ and $x_i'$ denote the polarizing variables of $x_i$ in $R^{\pol}$. Notice that, since $Q'$ is minimal, then not all $x_{3n-1}, x_{3n}, x_{3n+1}$ are in $Q'$. Moreover, there are at most $2$ of them in $Q'$. If one of these is not in $Q$, then we are done. Therefore, we consider the following cases:
\begin{itemize}
    \item  If $x_{3n}\in Q \setminus Q'$, then $x_{3n-1}, x_{3n+1}\in Q'$. Since $x_{3n+1}\in Q'\subseteq Q$ then $x_{3n+1}\in W$, since $\deg_{x_{3n+1}}(I_n)=1$. Thus, $x_{3n-1}\notin W$. So $x_{3n-1}'\in W$ and thus $ x_{3n-4}\notin W$. If $x_{3n-4}'\notin W$, then $x_{3n-4}\notin Q$, and we are done. Otherwise, if $x_{3n-4}'\in W$, then $x_{3n-7}\notin W$. If $x_{3n-7}'\notin W$, then $x_{3n-7}\notin Q$, and we are done again. Otherwise, if $x_{3n-7}'\in W$, then $x_{3n-10}\notin W$. Continuing this way, we may assume assume that $x_{3i-1}'\in W$ for $2\le i\le n$. In particular, $x_5'\in W$ and thus $x_2\notin W$. Since $x_{3n+1}\in W$, then $x_2'\notin W$. Thus, $x_2\notin Q$, as desired.
    \item  If $x_{3n+1}\in Q\setminus Q'$, then $\left\{x_2,x_{3n}, x_{3n+2}\right\}\subseteq Q'$. Since $x_{3n}\in Q'\subseteq Q$, then $x_{3n}\in Q$, and hence $x_{3n}\in W$. Since $x_{3n+1}\in Q$ then $x_{3n+1}\in W$. So $x_{3n-1}\notin W$. If $x_{3n-1}\notin Q$ then we are done. Otherwise, we have $x_{3n-1}'\in W$. If $x_{3n-4}\notin Q$, again, we are done. Otherwise, we have $x_{3n-4}'\in W$. Continuing this way, we may assume that $x'_{3i-1}\in W$ for $2\le i\le n$. In particular, $x_5'\in W$ and hence $x_2\notin W$. Since $x_2\in Q_n$, then $x_2\in Q$. Hence, $x_2'\in W$. On the other hand, since $x_{3n+1}\in W$, then $x_2'\notin W$, a contradiction.
    \item  If $x_{3n-1}\in Q\setminus Q'$, then $\{x_{3n}, x_{3n-3}, x_{3n-4}\}\subseteq Q'$. Since $x_{3n}\in Q'\subseteq Q$, then $x_{3n}\in Q$ and hence $x_{3n}\in W$. If $x_{3n+1}\notin Q$, then we are done. Otherwise, if $x_{3n+1}\in Q$, then $x_{3n+1}\in W$. Then $x_{3n-1}\notin W$. If $x_{3n-1}\notin Q$, then we are done. Otherwise, we have $x_{3n-1}'\in W$. So, $x_{3n-4}\notin W$. Since $x_{3n-4}\in Q'$, then $x_{3n-4}\in Q$. Hence $x_{3n-4}'\in W$. By a similar argument, we may assume that $x_{3i-1}'\in W$ for $2\le i\le n$. In particular, $x_5'\in W$ and hence $x_2\notin W$. If $x_2\notin Q$, again, we are done. Otherwise, $x_2'\in W$. Then $x_{3n+1}\notin W$, a contradiction. 
\end{itemize}
Therefore, $f_{n+1}\notin Q$ for any embedded associated prime $Q$ in $R/I_n$. Hence, $f_{n+1}$ is regular on $R/I_n$, completing the proof.
\end{proof}

Next, we will give an explicit formula to calculate the depth of certain unicyclic graphs defined below.

\begin{definition}\label{unicyclic def}
Let  $n\ge 3, m\ge 2$, $A_n = k[x_1,\ldots,x_n]$,  $B_m = k[y_1,\ldots,y_m]$, and $R_{n,m} = k[x_1,\ldots,x_n,y_1,\ldots,y_m]$.
Let $G_{n,m}$ denote the graph whose edge ideal is $$I(G_{n,m}) = (I(C_n), x_{2}y_1, I(P_m)),$$ where $C_n$ is the $n$ cycle on variables $x_1,\ldots, x_n$ and $P_m$ is a path on variables $y_1,\ldots,y_m$.  By convention, we denote $I(G_{n,0}) = I(C_n)$ and $I(G_{n,1}) = (I(C_n), x_2y_1)$. 
Below is a figure showing the graph $G_{n,m}$ for arbitrary $n, m$.

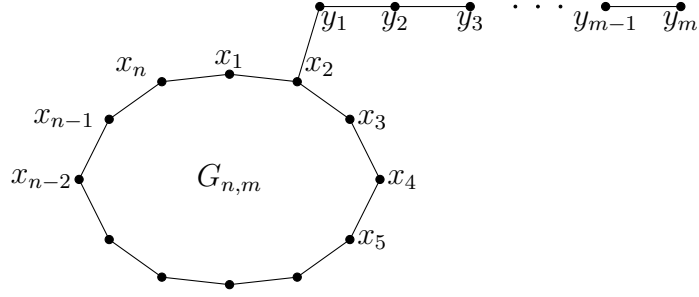
\begin{center}
  \begin{tikzpicture}
  \filldraw (0,3.8) circle(.3ex); 
  \filldraw (0,1) circle(.3ex);
    \filldraw (0.9,3.7) circle(.3ex);
    \filldraw (1.6,3.2) circle(.3ex); 
    \filldraw (-1.6,3.2) circle(.3ex);
    \filldraw (2,2.4) circle(.3ex); 
    \filldraw (-2,2.4) circle(.3ex);
    \filldraw (1.6,1.6) circle(.3ex);\filldraw (-1.6,1.6) circle(.3ex);
    \filldraw (0.9,1.1) circle(.3ex);
    \filldraw (-0.9,1.1) circle(.3ex);
    \filldraw (-0.9, 3.7) circle(.3ex);
    \filldraw (1.2, 4.7) circle(.3ex); \filldraw (2.2, 4.7) circle(.3ex);\filldraw (3.2, 4.7) circle(.3ex); \filldraw (3.8, 4.7) circle(.1ex);\filldraw (4.1, 4.7) circle(.1ex);\filldraw (4.4, 4.7) circle(.1ex);\filldraw (5, 4.7) circle(.3ex);\filldraw (6, 4.7) circle(.3ex);
    
    \draw (0.9, 3.7) -- (0,3.8);\draw (0.9, 3.7) -- (1.6,3.2);\draw (1.6,3.2) -- (2, 2.4);\draw (2, 2.4) -- (1.6,1.6);\draw (1.6, 1.6) -- (0.9,1.1);\draw (0.9,1.1) -- (0,1);\draw (-0.9,1.1) -- (0,1); \draw (-0.9,1.1) -- (-1.6,1.6); \draw (-1.6,1.6) -- (-2,2.4);\draw (-2,2.4) -- (-1.6,3.2);\draw (-1.6,3.2) -- (-0.9,3.7);\draw (-0.9,3.7) -- (0,3.8); \draw (0.9,3.7) -- (1.2, 4.7) ; \draw (1.2, 4.7) -- (2.2, 4.7); \draw (2.2,4.7)--(3.2,4.7); \draw (5,4.7)--(6,4.7); 
     \node at (0,2.4) {$G_{n, m}$};
    \node at (0, 4) {$x_1$};
    \node at (1.2, 3.9) {$x_2$}; \node at (1.9, 3.2) {$x_3$};\node at (2.3, 2.4) {$x_4$}; \node at (1.9, 1.6) {$x_5$};
    \node at (-2.5,2.4) {$x_{n-2}$}; 
    \node at (-2.2, 3.2) {$x_{n-1}$}; \node at (-1.3, 3.9) {$x_{n}$};  \node at (1.4, 4.5) {$y_1$};  \node at (2.2, 4.5) {$y_2$};  \node at (3.2, 4.5) {$y_3$};     \node at (5, 4.5) {$y_{m-1}$};  \node at (6, 4.5) {$y_m$};   
    \end{tikzpicture}
    \end{center}
\end{definition}

To give a formula for the depth of $R_{n,m}/I(G_{n,m})$ we will make use of induction. The next proposition gives the depth of the edge ideal of the unicyclic graph $G_{n,1}$ for any $n\ge 3$.

\begin{proposition}\label{G_{n,1}}
  For $n\ge 3$,    $\depth(R_{n,1}/I(G_{n,1})) = \lceil \dfrac{n}{3}\rceil$. 
\end{proposition}
\begin{proof} Let $I = I(G_{n,1})$. Consider the following short exact sequence
    $$0 \longrightarrow R_{n,1}/(I : x_2) \longrightarrow R_{n,1}/I \longrightarrow R_{n,1}/(I, x_2) \longrightarrow 0.$$
    Let $J_1 = (I : x_2)$ and $J_2 = (I , x_2)$.
    We see that $J_1 = (I : x_2) = (J_1',x_1,x_3,y_1)$,
    where $J_1'=(x_4x_5, x_5x_6, \ldots, x_{n-1}x_n)$ is the edge ideal of a path. 
    Hence
    \begin{align*}
        \depth(R_{n,1}/J_1)        &= \depth\left(\dfrac{k[x_4,\ldots, x_n]}{J_1'}\right)[x_2]\\
        &= \lceil\dfrac{n-3}{3}\rceil+ 1= \lceil\dfrac{n}{3}\rceil, \quad \text{by \cite[Lemma~2.8]{M-depthoftree}.}
       \end{align*}
    Next notice that $J_2 = (J_2', x_2)$,
    where $J_2'$ is the edge ideal of path on $x_3,x_4,\ldots,x_{n-1},x_n,x_1$. Thus 
    \begin{align*}
        \depth(R_{n,1}/J_2)         &= \depth\dfrac{k[x_1,x_3,x_4,\ldots,x_n]}{J_2'} + 1\\
        &= \lceil\dfrac{n - 1}{3}\rceil + 1= \lceil\dfrac{n+2}{3}\rceil, \quad \text{by \cite[Lemma~2.8]{M-depthoftree}.}
    \end{align*} 
   Since $\depth(R_{n,1}/J_2) \ge \depth(R_{n,1}/J_1)$, then by \cite[Theorem~4.3]{depthmodulo} we have that $$\depth(R_{n,1}/I) = \depth(R_{n,1}/J_1) = \lceil\dfrac{n}{3}\rceil,$$
   as claimed.
\end{proof}

Before we can prove the formula for the depth of $R_{n,2}/I(G_{n,2})$ we show first a lower bound in a special case. 

\begin{lemma}\label{G_{3n+2,2}}
For $t\in\mathbb{N}$, $\depth(R_{3t+2,2}/I(G_{3t+2,2}) \ge t+2$.    
\end{lemma}
\begin{proof} 
Let $f_1 = x_1+x_{3t+2} + x_2$,  $f_i=x_{3i-2}+x_{3i-3}+x_{3i-1}$ for $2\le i\le t $,  $f_{t+1}=x_{3t} + x_{3t+1} + \sum \limits_{i=1}^{t}x_{3i-1}$ and $f_{t+2} = y_2 + y_1$. Using an argument similar to the proof of Theorem~\ref{C_{3n+2}} one can see that $f_1, \ldots, f_{t+1}$ is an initially regular sequence on $R/I(G_{3t+2,2})$ with respect to a term order such that $x_1 > x_{3t+2} > x_2$, and $x_{3i-2}>x_{3i-3}>x_{3i-1}$, for all $2\le i\le t $. It remains to show that $f_{t+2}$ is regular on $I_{t+1}$, where 
$I_{i+1} = \ini(I_i, f_i)$ for $1 \le i \le t$ and $I_1 = \ini(I) = I$. 

By \cite[Lemma~3.7]{FHM-ini}, we see that $d_{y_1}(I_{t+1}) = 1$. Next, we will show that $y_1y_2$ is the only monomial generator that $y_2$ divides in $I_{t+1}$.  We will proceed by induction. Notice that $y_1y_2 \in I$ and it is the only monomial that $y_2$ divides in $I$. Now, for $1 \le i\le t$, suppose that $y_1y_2\in I_i$ and is the only monomial that $y_2$ divides in $I_i$. Consider $I_{i+1} = \ini(I_i, f_{i+1})$. Let $R_2 = k[x_{3i+1}, x_{3i}, x_{3i+2}]$ if $i<t$ and let $R_2 = k[x_{3t}, x_{3t+1}, x_{3j-1} \mid 2\le j\le t]$ if $i=t+1$. Let $R_1$ be the polynomial ring such that $R = R_1[x_{3i+1}, x_{3i}, x_{3i+2}]$ if $i<t$ and if $i=t+1$ let $R_1$ be the polynomial ring such that $R=R_1[x_{3t}, x_{3t+1}, x_{3j-1} \mid 2\le j\le t]$. It is easy to see that $I_{i+1}$ is $(R_1,R_2)$-factorable in the sense of \cite[Definition~3.2]{FHM-ini}. By \cite[Proposition~3.5]{FHM-ini}, the reduced Gr\"{o}bner basis of $I_{i+1}$ is $(R_1,R_2)$-factorable  and consists of elements of the form $\lcm(M_{i_1}, \ldots, M_{i\ell})g$, where $M_{i_1}, \ldots, M_{i_{\ell}}\in R_1$ are monomials such that $M_{i_j}$ divides a monomial generator of $I_{i}$ and $g\in R_2$. Let $N\in I_{i+1}$ be monomial such that $y_2 \mid N$. Notice that $N = \lcm(M_1,\ldots, M_p)\ini(g)$ for some  $M_1, \ldots, M_p\in R_1$ monomials that divide monomial generators of $I_i$ and $g\in R_2$. Since $g\in R_2$, then $\ini(g)\in R_2$ and hence $y_2 \mid \lcm(M_1, \ldots, M_p)$. Thus $y_2\mid M_j$ for some monomial $M_j$ in $R_1$. In other words, $M_j$ is a monomial in $I_i$ such that $y_2 \mid M_j$, and hence, by induction $M_j=y_1y_2$, see also the proof of \cite[Proposition~3.5]{FHM-ini}. So, $N$ is a multiple of $y_1y_2 \in I_i$, and thus for $N$ to be in the reduced Gr\"{o}bner basis of $I_{i+1}$ it must be that $N=y_1y_2$.

Since the only monomial that $y_2$ divides in $I_{t+1}$ is $y_1y_2$, it follows by \cite[Lemma~3.9]{FHM-ini} that $f_{t+2} = y_2 + y_1$ is regular on $R_{3t+2,2}/I_{t+1}$ with respect to an appropriate order as in the statement. Therefore, $f_1,\ldots,f_{t+2}$ is an initially regular sequence on $R/I$, with respect to aforementioned order and hence $\depth(R_{3t+2,2}/I)\ge t+2$.
\end{proof}

Next, we prove a general formula in the case of the unicyclic graph $G_{n,2}$ for any $n\ge 3$. 

\begin{proposition} \label{m=2 case}
    For $n\ge 3$,    $\depth(R_{n,2}/I(G_{n,2})) = \lceil \dfrac{n-1}{3}\rceil+1$. 
\end{proposition}

\begin{proof}
    Let $I = I(G_{n,2})$ for $n\ge 3$. Consider the following short exact sequence
    $$0 \longrightarrow R_{n,2}/(I : y_2) \longrightarrow R_{n,2}/I \longrightarrow R_{n,2}/(I, y_2) \longrightarrow 0.$$
    Let $J_1 = (I : y_2)$ and $J_2 = (J,y_2)$. Notice that $J_1 = (I:y_2) = (I(C_n) , y_1)$ and hence
    \begin{align*}
        \depth(R_{n,2}/J_1) &= \depth\dfrac{k[x_1,\ldots,x_n]}{I(C_n)}[y_2]\\
        &= \depth(A_n/I(C_n)) + 1\\
        &= \lceil\dfrac{n-1}{3}\rceil + 1= \lceil\dfrac{n+2}{3}\rceil, \quad \text{by Theorem }\ref{Jacq}.
    \end{align*}
    Next, notice that $J_2 = (I,y_2) = (I(G_{n,1}),y_2)$ and thus
    \begin{align*}
        \depth(R_{n,2}/J_2) &=  \depth\dfrac{k[x_1,\ldots,x_n]}{I(G_{n,1})} [y_1]= \lceil\dfrac{n}{3}\rceil, \quad \mbox{by Proposition }~\ref{G_{n,1}}.
    \end{align*}
    If $n\equiv 1 \mod 3$, then $\lceil\dfrac{n+2}{3}\rceil=\lceil\dfrac{n}{3}\rceil=\lceil\dfrac{n-1}{3}\rceil+1$ and in that case $$\depth R_{n,2}/I=\lceil\dfrac{n-1}{3}\rceil+1,$$ by \cite[Theorem~4.3]{depthmodulo}. We now handle the remaining two cases. 
    
    If $n=3t$ with $t\in\mathbb{N}$, we have $\depth(R_{n,2}/J_1) = t+1$ and $\depth(R_{n,2}/J_2) = t$. Let $f_1 = x_1 + x_{3t} + x_2, f_i =x_{3i-2}+x_{3i-3}+x_{3i-1}$ for all $2\le i\le t$, and  $f_{t+1} = y_2+y_1$. By \cite[Theorem~3.11]{FHM-ini}, $f_1, \ldots, f_{t+1}$ is an initially regular sequence on $R_{n,2}/I$ with respect to an order such that $x_1 > x_{3t} > x_2$, $x_{3i-2}>x_{3i-3}>x_{3i-1}$ for all $2\le i \le t$ and  $y_2>y_1$. Thus, $\depth(R_{n,2}/J)\ge t+1$ and thus $$\depth(R_{n,2}/J) = \depth(R_{n,2}/J_1)=t+1=\lceil\dfrac{n+2}{3}\rceil=\lceil\dfrac{n-1}{3}\rceil+1,$$ by \cite[Theorem~4.3]{depthmodulo}.
       
    In the remaining case that $n=3t+2$ with $t\in\mathbb{N}$ we have $\depth(R_{n,2}/J_1) = t+2$ and $\depth(R_{n,2}/J_2) = t+1$. Moreover, by Lemma~\ref{G_{3n+2,2}}, we have $\depth(R_{n,2}/J)\ge t+2$. Hence $\depth(R_{n,2}/J)\ne\depth(R_{n,2}/J_2)$ and therefore,  
    $$\depth(R_{n,2}/J) = \depth(R_{n,2}/J_1) = \depth(A_n/I(C_n)) + 1 = \lceil\dfrac{n-1}{3}\rceil + 1,$$
   by \cite[Theorem~4.3]{depthmodulo}. 
\end{proof}

In the final result of this paper, we establish the depth of $R_{n,m}/I(G_{n,m})$ for any $m\ge 0$ and any $n\ge 3$.
\begin{theorem}\label{G_nm}
    For $n\ge 3$ and $m\ge 0$, we have that
    \[ \depth(\frac{R_{n,m}}{I(G_{n,m})}) = \left\{
\begin{array}{lc}
      \depth(\dfrac{A_n}{I(C_{n})}) + \dfrac{m}{3}=\lceil\dfrac{n-1}{3}\rceil + \dfrac{m}{3}, & \quad \mbox{if } m\equiv 0 \mod 3, \\
      \lceil\dfrac{n}{3}\rceil + \dfrac{m-1}{3}=\lceil\dfrac{n}{3}\rceil + \dfrac{m-1}{3}, & \quad \mbox{if } m\equiv 1\mod 3,\\
      \depth(\dfrac{A_n}{I(C_n)}) + \dfrac{m+1}{3}=\lceil\dfrac{n-1}{3}\rceil + \dfrac{m+1}{3}, & \quad \mbox{if } m\equiv 2\mod 3.
\end{array}
\right. \]  
\end{theorem}
    \begin{proof} We proceed by induction on $m$.  When $m=0$ there is nothing to show. The case when $m=1$ is shown in Proposition~\ref{G_{n,1}}, whereas the case $m=2$ is shown in Proposition~\ref{m=2 case}.

Let  $n\ge 3, m\ge 3$ and let $I = I(G_{n,m})$. Consider the following short exact sequence
    $$0 \longrightarrow R_{n,m}/(I : y_{m-1}) \longrightarrow R_{n,m}/I \longrightarrow R_{n,m}/(I, y_{m-1}) \longrightarrow 0.$$
    Let $I_1 = (I:y_{m-1})$ and $I_2 = (I,y_{m-1})$. Notice that $I_1 = (I(G_{n,m-3}),y_{m-2},y_m)$ and thus 
      \begin{align*}
        \depth(R_{n,m}/I_1) = \depth (\dfrac{k[x_1,\ldots,y_{m-3}]}{I(G_{n,m-3})}[y_{m-1}])=\depth(R_{n,m-3}/I(G_{n,m-3})) + 1.
     \end{align*}
     On the other hand, we have $I_2 = (I(G_{n,m-2}),y_{m-1})$ and thus 
     \begin{align*}
        \depth(R_{n,m}/I_2) =\depth (\dfrac{k[x_1,\ldots,x_n,y_1,\ldots,y_{m-2}]}{I(G_{n,m-2})}[y_m])= \depth(R_{n,m-2}/I(G_{n,m-2})) + 1.
     \end{align*}
     Write $m=3h+r$, where $0\le r\le 2$ and $h\in \mathbb{N}$, and consider all three cases separately. 

    If $r=0$, then  $m=3h$ and by the inductive hypothesis
         \begin{align*}
             \depth(R_{n,m}/I_1) &= \depth(R_{n,3h-3}/I(G_{n,3h-3}) + 1\\
             &= \lceil\dfrac{n-1}{3}\rceil + \dfrac{3h-3}{3} + 1= \lceil\dfrac{n-1}{3}\rceil + h.
         \end{align*}
Also, 
\begin{align*}
             \depth(R_{n,m}/I_2) &= \depth(R_{n,3h-2}/I(G_{n,3h-2}) + 1\\
             &= \lceil\dfrac{n}{3}\rceil + \dfrac{3h-2-1}{3} +1= \lceil\dfrac{n}{3}\rceil + h.
         \end{align*}
         Therefore, $\depth(R_{n,m}/I_2)\ge\depth(R_{n,m}/I_1)$, and hence 
         $$\depth(R_{n,m}/I) =\depth(R_{n,m}/I_1) = \lceil\dfrac{n-1}{3}\rceil+ h = \lceil\dfrac{n-1}{3}\rceil + \dfrac{m}{3},$$
         by \cite[Theorem~4.3]{depthmodulo}.

          If $r=1$, then $m=3h+1$ and by the inductive hypothesis
         \begin{align*}
             \depth(R_{n,m}/I_1) &= \depth(R_{n,3h-2}/I(G_{n,3h-2})) + 1\\
             &= \lceil\dfrac{n}{3}\rceil + \dfrac{3h-2-1}{3} + 1= \lceil\dfrac{n}{3}\rceil + h,         
         \end{align*}
         and
         \begin{align*}
             \depth(R_{n,m}/I_2) &= \depth(R_{n,3h-1}/I(G_{n,3h-1})) + 1\\
             &= \lceil\dfrac{n-1}{3}\rceil + \dfrac{3h-1+1}{3} +1= \lceil\dfrac{n+2}{3}\rceil + h.
         \end{align*}
        Hence,  $\depth(R_{n,m}/I) = \depth(R_{n,m}/I_1) = \lceil\dfrac{n}{3}\rceil + h = \lceil\dfrac{n}{3}\rceil + \dfrac{m-1}{3}$, by \cite[Theorem~4.3]{depthmodulo}.

          If $r=2$, then  $m=3h+2$ and again by the inductive hypothesis we have
         \begin{align*}
             \depth(R_{n,m}/I_1) &= \depth(R_{n,3h-1}/I(G_{n,3h-1}) + 1\\
             &= \lceil\dfrac{n-1}{3}\rceil + \dfrac{3h-1+1}{3} + 1= \lceil\dfrac{n-1}{3}\rceil + h + 1,
         \end{align*}
         and
         \begin{align*}
             \depth(R_{n,m}/I_2) &= \depth(R_{n,3h}/I(G_{n,3h}) + 1\\
             &= \lceil\dfrac{n-1}{3}\rceil + \dfrac{3h}{3} +1= \lceil\dfrac{n-1}{3}\rceil + h + 1.
         \end{align*}
         Therefore, as before by \cite[Theorem~4.3]{depthmodulo} we have $$\depth(R_{n,m}/I) = \depth(R_{n,m}/I_1) = \lceil\dfrac{n-1}{3}\rceil + h + 1= \lceil\dfrac{n-1}{3}\rceil +  \dfrac{m+1}{3}.$$
    This completes the proof.
 \end{proof}

\section{Acknowledgement}
I would like to express my deepest gratitude to my advisor, Professor Louiza Fouli, for her expert guidance and constant encouragement. I also want to thank Professor T\`{a}i H\`{a} for his suggestion of the question regarding the depth of unicyclic graphs, and Professor Susan Morey for her comments about initially regular sequences on cycles in general.

\end{document}